 \documentclass[11pt,final]{siamltex}
\usepackage{amsmath}
\usepackage{graphicx}% Figure
\usepackage{subfigure}% Subfigure
 \usepackage{threeparttable}
 \usepackage{algorithm,version}  
\usepackage{algorithmic}

\usepackage{bm}
\usepackage{amssymb}
\usepackage{cases}
\usepackage{color}
\newtheorem{remark}{Remark}[section]

    \title{Efficient Linear and Unconditionally Energy Stable  Schemes for the Modified Phase Field Crystal Equation\thanks{The work of X. Li is supported by the National Natural Science Foundation of China under grant number 11901489, 11971407 and Postdoctoral Science Foundation of China under grant numbers BX20190187 and 2019M650152. The work of J. Shen is supported in part by NSF grant  DMS-1720442 and AFOSR  grant FA9550-16-1-0102.}
}
    \author{ Xiaoli Li
        \thanks{School of Mathematical Sciences and Fujian Provincial Key Laboratory on Mathematical Modeling and High Performance Scientific Computing, Xiamen University, Xiamen, Fujian, 361005, China. Email: xiaolisdu@163.com}.
        \and Jie Shen 
         \thanks{%Corresponding Author. 
         Department of Mathematics, Purdue University, West Lafayette, IN 47907, USA. Email: shen7@purdue.edu}.
}

\begin{document}

\maketitle

\begin{abstract}
In this paper, we construct efficient schemes based on the scalar auxiliary variable (SAV) block-centered finite difference method for the modified phase field crystal (MPFC) equation, which is a sixth-order nonlinear damped wave equation. The schemes are linear, conserve mass  and unconditionally dissipate a pseudo energy. We prove  rigorously second-order error estimates in both time and space  for the phase field variable in discrete norms. We also present some numerical experiments  to verify our theoretical results and demonstrate the robustness and accuracy.
\end{abstract}

 \begin{keywords}
 Modified phase field crystal, scalar auxiliary variable (SAV),  energy stability, error estimate, numerical experiments
 \end{keywords}
 
 \begin{AMS}
35G25, 65M06, 65M12, 65M15
    \end{AMS}
\pagestyle{myheadings}
\thispagestyle{plain}
\markboth{XIAOLI LI AND JIE SHEN} {Energy Stability and Convergence for MPFC Model}
%==================================================================
 \section{Introduction}
 %Crystallization  plays an important role in the purification of solid compounds. If a saturated hot solution is allowed to cool off, the solute is no longer soluble in the solvent and forms crystals of pure compound. During this process, the crystal growth is the major stage in crystallization. 
 The phase field crystal (PFC) model was developed in \cite{elder2004modeling,elder2002modeling} to model the crystallization process in the purification of solid compounds. It has been used  to model the evolution of the atomic-scale crystal growth on diffusive time scales. In the PFC model, the phase field variable is introduced to describe the phase transition from the liquid phase to the crystal phase. The model is  versatile and able  to simulate various phenomena, such as grain growth, epitaxial growth, reconstructive phase transitions, material hardness, and crack propagations.
 Numerical methods and simulations for the PFC model have been studied extensively, including  finite element method \cite{gomez2012unconditionally}, finite difference methods \cite{li2017efficient,wise2009energy,zhang2013adaptive}, local discontinuous Galerkin method \cite{guo2016local} and Fourier-spectral method \cite{yang2017linearly}. 
 
 The modified phase field crystal (MPFC) equation was introduced in \cite{stefanovic2006phase} to model phase-field crystals with elastic interactions.  The MPFC equation can be viewed as a perturbed gradient flow with respect to a free energy, and is a sixth order  nonlinear damped wave equation.  However, as pointed out in \cite{wang2010global}, the original free energy of the MPFC equation may increase in time on some time intervals. Thus  A pseudo energy is introduced in \cite{wang2010global} and shown to  be dissipative. 
There exist a number of work  on the numerical approximations of the MPFC model. 
%Wang and Wise \cite{wang2010global} studied the existence and uniqueness of global smooth solutions of the MPFC equation by assuming that the initial data are smooth.
 First and second order accurate nonlinear  convex splitting schemes have been proposed in \cite{wang2010global,baskaran2013convergence}, and are proved to be  unconditional energy stable and  convergent.  A nonlinear multigrid method is used to solve the nonlinear system at each time step \cite{baskaran2013energy}. Guo and Xu \cite{guo2018high}  developed a first-order and a second-order nonlinear convex splitting, and a first-order linear energy stable fully discrete with local discontinuous Galerkin (LDG) methods. Very recently, Li and his coauthors 
\cite{li2019efficient} proposed  unconditional energy stable schemes based on the "Invariant Energy Quadratization" (IEQ) approach for the MPFC model but without  convergence proof. The  convergence analysis is  challenging due to the nonlinear hyperbolic nature of the MPFC equation. To our knowledge, there is  no second-order convergence analysis  on any linear scheme for the MPFC equation.

The main goals of this paper are to construct linear and unconditionally energy stable schemes based on the recently proposed scalar auxiliary variable (SAV) approach \cite{shen2018scalar,shen2019new}, and to carry out a rigorous error analysis. More specifically, we construct two SAV block-centered finite difference schemes for the MPFC equation based on the Euler backward and Crank-Nicolson schemes respectively, and show that they are unconditionally energy stable with a suitably defined pseudo energy, and we establish second-order convergence in both time and space in a discrete $L^{\infty}(0,T;H^3(\Omega))$ norm.
% The primary contribution of our paper is that we obtain the second-order accuracy in the temporal and spatial meshsize by combining the hyperbolic nature of MPFC equation and the novel analytical technique in the SAV approach inspired by our recent work \cite{li2019energy}.
%Finally, numerical experiments are presented demonstrating the capability and efficiency of the proposed methods.

The rest of the paper is organized as follows. In Section 2 we describe  the MPFC model and reformulate it  using the SAV approach.  In Section 3 we construct  fully discrete schemes for the reformulated MPFC equation by  block-centered finite difference method, and show that the scheme  conserves mass and is unconditionally energy stable. In Section 4 we derive the error estimate for the MPFC model. In Section 5 some numerical experiments are presented to verify the accuracy of the proposed numerical schemes. 
 
%==================================================================
\section{The MPFC model and its semi-discretization in time} \label{sec:Preliminary}
We describe in this section the MPFC model, its reformulation using the SAV approach,  construct a second-order SAV semi-discretization scheme and show that it preserves mass and  dissipates a pseudo energy.

\subsection{The MPFC model and its SAV reformulation}
Consider the free energy (cf. \cite{baskaran2013energy,baskaran2013convergence,guo2018high})
\begin{equation}\label{e_definition of free energy}
\aligned
E(\phi)=\int_{\Omega}\{\frac{1}{2}(\Delta \phi)^2-|\nabla \phi|^2+\frac \alpha 2 \phi^2+F(\phi)\}d\textbf{x},
\endaligned
\end{equation}
where $\Omega\subset \mathbb{R}^d \,(d=1,2,3)$. The phase field variable $\phi$ is introduced to represent the concentration field of a coarse-grained temporal average of the density of atoms. $F(\phi)=\frac 1 4 \phi^4$. Here  $\alpha=1-\epsilon$ with $\epsilon\ll1$.  Then the MPFC model, which is designed to describe the elastic interactions:
\begin{equation}    \label{e_continuous_model}
\left\{
\begin{array}{l}
\displaystyle \frac{\partial^2 \phi}{\partial t^2}+\beta\frac{\partial \phi}{\partial t}=M\Delta\mu, \ \ \textbf{x}\in \Omega, t>0,\\
\displaystyle \mu=\Delta^2 \phi+2\Delta \phi+\alpha\phi+F^{\prime}(\phi), \ \ \textbf{x}\in \Omega, t>0,\\
\displaystyle \phi(\textbf{x},0)=\phi_0(\textbf{x}),
\end{array}
\right.
\end{equation}
where $\beta> 0$. But we should note that the energy \eqref{e_definition of free energy} may actually increase on some time intervals. The PFC and MPFC equations have close relationship. However, we should keep in mind that the original energy of the MPFC equation may increase in time on some time intervals. Thus it is desirable to introduce a pseudo energy. Besides, we can observe that  \eqref{e_continuous_model} does not satisfy the mass conservation due to the term $\frac{\partial^2 \phi}{\partial t^2}$. However, it is possible to verify that $\int_{\Omega}\frac{\partial \phi}{\partial t}d\textbf{x}=0$ with a suitable initial condition for $\frac{\partial \phi}{\partial t}$.

To fix the idea, we consider the homogeneous Neumann boundary conditions:
\begin{equation}\label{e_boundary and initial condition}
   \partial_\textbf{n}\phi|_{\partial \Omega}=0,\  \partial_\textbf{n}\Delta \phi|_{\partial \Omega}=0, \ 
    \partial_\textbf{n}\mu|_{\partial \Omega}=0,
  \end{equation}
%  In the following arguments, we will only consider the periodic boundary condition for convenience. For the case of homogenous Neumann boundary conditions, as long as AAA is replaced by BBB, all theoretical results are still valid.
  where $\textbf{n}$ is the unit outward normal vector of the domain $\Omega$.

\begin{remark}
The homogeneous Neumann boundary conditions are assumed to simplify the presentation.  The algorithm and  its analysis also hold for the periodic boundary conditions with very little  modification. One can refer to \cite[Lemma 3.6]{wang2011energy} for more detail about the periodic boundary conditions. 
While we only present the algorithm and analysis for homogeneous Neumann boundary conditions, we do   present some numerical results with  periodic boundary conditions in Section 5.
\end{remark}

To introduce an appropriate pseudo energy for the MPFC equation, we need to define the $H^{-1}$ inner-product \cite{baskaran2013energy}. Let $u_i \,(i=1,2) \in \{f\in L^2(\Omega)|\ \int_{\Omega}fd\textbf{x}=0\}:=L_0^2(\Omega)$, we define $\eta_{u_i}\in H^2(\Omega)\cap L_0^2(\Omega)$ to be the unique solution to the following problem:
\begin{equation}\label{e_pseudo_energy1}
\aligned
-\Delta \eta_{u_i}=u_i \ \ in \ \Omega, \ \ \partial_\textbf{n} \eta_{u_i} |_{\partial \Omega}=0.
\endaligned
\end{equation} 
Then we have $\eta_{u_i}=-\Delta^{-1}u_i$. Define 
\begin{equation}\label{e_pseudo_energy2}
\aligned
(u_1,u_2)_{H^{-1}}:=(\nabla\eta_{u_1},\nabla\eta_{u_2})_{L^2}.
\endaligned
\end{equation} 
 Using integration by parts, we can obtain
\begin{equation}\label{e_pseudo_energy3}
\aligned
(u_1,u_2)_{H^{-1}}=-(\Delta^{-1}u_1,u_2)_{L^2}=-(\Delta^{-1}u_2,u_1)_{L^2}=(u_2,u_1)_{H^{-1}}.
\endaligned
\end{equation}
Then we define $\|u\|_{H^{-1}}=\sqrt{(u,u)_{H^{-1}}}$ for every $u\in L^2_0(\Omega)$.

%Introducing three auxiliary functions as follows:
%\begin{equation}\label{e_auxiliary1}
%\aligned
%\psi=\frac{\partial \phi}{\partial t},\ \xi=\Delta \phi,\ R=\sqrt{E_1(\phi)+C_0},
%\endaligned
%\end{equation} 
%where $E_1(\phi)=\int_{\Omega}F(\phi)d\textbf{x}$ and $C_0$ is a reasonable constant to guarantee that $E_1(\phi)+C_0>0$.

In order to construct an efficient scheme for the MPFC equation \eqref{e_continuous_model}, we first reformulate it using the so called SAV approach \cite{shen2018scalar}.  Introducing two auxiliary functions as follows:
\begin{equation}\label{e_auxiliary1}
\aligned
\psi=\frac{\partial \phi}{\partial t},\ r=\sqrt{E_1(\phi)}:=\sqrt{\int_{\Omega}F(\phi)d\textbf{x}}.
\endaligned
\end{equation} 
Then the MPFC equation \eqref{e_continuous_model} can be recast as the following system:
  \begin{subequations}\label{e_model_recast}
    \begin{align}
    &\frac{\partial \psi}{\partial t}+\beta\psi=M\Delta \mu,   \label{e_model_recastA}\\
   & \mu=\Delta^2 \phi+2\Delta \phi+\alpha\phi+\frac{r(t)}{\sqrt{E_1(\phi)}}F^{\prime}(\phi),   \label{e_model_recastB}\\
    &r_t=\frac{1}{2\sqrt{E_1(\phi)}}\int_{\Omega}F^{\prime}(\phi)\phi_t d\textbf{x}. \label{e_model_recastC}
    \end{align}
  \end{subequations}
  Define the pseudo energy
  \begin{equation}\label{e_pseudo_energy}
\aligned
\mathcal{E}(\phi,r,\psi)=\int_{\Omega}(\frac{1}{2}(\Delta \phi)^2-|\nabla \phi|^2+\frac \alpha 2 \phi^2)d\textbf{x}+r^2+\frac{1}{2M}\|\psi\|_{H^{-1}}^2,
\endaligned
\end{equation} 
which requires that $\int_{\Omega}\psi=0$ for well posedness. 
%For the sake of safety, we use the initial data
%  \begin{equation}\label{e_initial data}\aligned
%\psi(\textbf{x},0)=\frac{\partial \phi}{\partial t}(\textbf{x},0)\equiv0.\endaligned
%\end{equation} 
%Thus it is obvious that $\int_{\Omega}\psi=0$.
 As long as $\psi=\frac{\partial \phi}{\partial t}$ is of mean zero, we can obtain the following 
 dissipation law:
\begin{equation}\label{e_evolution of pseudo energy}
\aligned
\frac{d }{dt}\mathcal{E}(\phi,\xi,r,\psi)=&\int_{\Omega}\mu\frac{\partial \phi}{\partial t}d\textbf{x}-\frac 1 M\int_{\Omega}\Delta \eta_{\psi}\frac{\partial \eta_{\psi}}{\partial t}d\textbf{x}\\
=&(\mu,\frac{\partial \phi}{\partial t})-\frac 1 M(\psi,\Delta^{-1}\frac{\partial \psi}{\partial t})\\
=&\frac{\beta}{M}(\psi,\Delta^{-1}\psi)=-\frac{\beta}{M}\|\psi\|_{H^{-1}}^2\leq 0,
\endaligned
\end{equation} 
where $\eta_{\psi}=(-\Delta )^{-1}\psi$. 

 \subsection{A second-order semi-discrete scheme} 
Let $N>0$ be a positive integer and $J=(0,T]$. Set $\Delta t=T/N,\ t^n=n\Delta t,\ \ \rm{for} \ n\leq N,$ where $T$ is the final time. The second-order semi-discrete scheme based on the Crank-Nicolsion method for \eqref{e_model_recast} is as follows:

Assuming $\phi^n$, $\psi^n$ and $r^n$ are known, then we update $\phi^{n+1}$, $\psi^{n+1}$ and $r^{n+1}$ by solving

 \begin{numcases}{}
 \psi^{n+1}-\psi^{n}+\beta\Delta t\psi^{n+1/2}=M\Delta t\Delta \mu^{n+1/2}, \label{e_semi_second-order1}\\
 \Delta t\psi^{n+1/2}=\phi^{n+1}-\phi^n, \label{e_semi_second-order2}\\
 \mu^{n+1/2}=\Delta^2 \phi^{n+1/2}+2\Delta \tilde{\phi}^{n+1/2}+\alpha \phi^{n+1/2}\notag\\
 \ \ \ \ \ +\frac{r^{n+1/2}}{\sqrt{E_1 (\tilde{\phi}^{n+1/2})}}F^{\prime}(\tilde{\phi }^{n+1/2}),  \label{e_semi_second-order3}\\
r^{n+1}-r^n=\frac{1}{2\sqrt{E_1 (\tilde{\phi }^{n+1/2}) }}(F^{\prime}(\tilde{\phi }^{n+1/2}),
\phi^{n+1}-\phi^n),  \label{e_semi_second-order4}
\end{numcases}
where $f^{n+1/2}=(f^{n+1}+f^{n})/2$ and  $\tilde{f }^{n+1/2}=(3 f ^n-f^{n-1})/2$ for any function $f$. For the case of $n=0$, we can computer $\tilde{\phi }^{1/2}$ by the first-order scheme.

%==================================================================
 %\subsection{Mass conservation and unconditional energy stability}  
%In this section, we derive the mass conservation and unconditional energy dissipation law  for the second-order semi-discrete scheme \eqref{e_semi_second-order1}-\eqref{e_semi_second-order4} with defined pseudo energy \eqref{e_pseudo_energy}.
\begin{theorem}\label{thm: semi-discrete mass conservation}
The scheme (\ref{e_semi_second-order1})-(\ref{e_semi_second-order4}) is mass conserving, i.e., $ \int_{\Omega} \phi^{n+1} d\textbf{x}=  \int_{\Omega} \phi^{n} d\textbf{x} $ for all $n$, and unconditionally stable in the sense that
\begin{equation}\label{semi-discrete energy stability1}
\aligned
\tilde{\mathcal{E}}(\phi^{n+1},r^{n+1},\psi^{n+1})-\tilde{\mathcal{E}}(\phi^n,r^n,\psi^n)\leq -\frac{\beta}{M}\Delta t\|\psi^{n+1/2}\|_{H^{-1}}^2,
\endaligned
\end{equation}
where $\tilde{\mathcal{E}}(\phi^n,r^n,\psi^n)=\mathcal{E}(\phi^n,r^n,\psi^n)+\frac 1 2\|\nabla \phi^{n}-\nabla \phi^{n-1}\|^2$.
\end{theorem}

\begin{proof}
Taking the inner products of \eqref{e_semi_second-order1} with $1$ leads to 
\begin{equation}\label{semi-discrete mass conservation1}
\aligned
(\psi^{n+1}-\psi^{n},1)+\beta\Delta t(\psi^{n+1/2},1)=M\Delta t(\Delta \mu^{n+1/2},1).
\endaligned
\end{equation}
Similarly, by taking the inner products of \eqref{e_semi_second-order2} with $1$, we can obtain
\begin{equation}\label{semi-discrete mass conservation2}
\aligned
(\phi^{n+1}-\phi^n,1)=\Delta t(\psi^{n+1/2},1).
\endaligned
\end{equation}
Using the integration by parts, the term on the right hand side of \eqref{semi-discrete mass conservation1} can be transformed into
\begin{equation}\label{semi-discrete mass conservation3}
\aligned
M\Delta t(\Delta \mu^{n+1/2},1)=-M\Delta t\left((\nabla \mu^{n+1/2},\nabla 1)+(\nabla \mu^{n+1/2}, \nabla 1) \right)=0.
\endaligned
\end{equation}
Then \eqref{semi-discrete mass conservation1} can be recast as follows:
\begin{equation}\label{semi-discrete mass conservation4}
\aligned
(1+\frac \beta 2\Delta t)(\psi^{n+1},1)=(1-\frac \beta 2\Delta t)(\psi^{n},1).
\endaligned
\end{equation}
Combining \eqref{semi-discrete mass conservation4} with the condition on the initial condition $(\psi^0,1)=0$ leads to $(\psi^{n+1},1)=0$ for all $n\leq 0$. Recalling \eqref{semi-discrete mass conservation2}, we have 
$(\phi^{n+1},1)=(\phi^{n},1)$.

Next, we prove \eqref{semi-discrete energy stability1}.
Taking the inner products of \eqref{e_semi_second-order2} with $\mu^{n+1/2}$ gives
\begin{equation}\label{semi-discrete energy stability2}
\aligned
\Delta t(\psi^{n+1/2},\mu^{n+1/2})=(\phi^{n+1}-\phi^n,\mu^{n+1/2}).
\endaligned
\end{equation}
Taking the inner products of \eqref{e_semi_second-order3} with $\phi^{n+1}-\phi^n$, we have
\begin{equation}\label{semi-discrete energy stability3}
\aligned
&(\mu^{n+1/2},\phi^{n+1}-\phi^n)=(\Delta^2 \phi^{n+1/2}, \phi ^{n+1}-\phi^{n}) +2(\Delta \tilde{ \phi }^{n+1/2}, \phi^{n+1}-\phi^n)\\
&\ \ \ \ \ +\alpha( \phi ^{n+1/2}, \phi^{n+1}-\phi^{n}) +(\frac{r^{n+1/2}}{\sqrt{E_1(\tilde{ \phi }^{n+1/2})}}F^{\prime}(\tilde{ \phi }^{n+1/2}), \phi ^{n+1}-\phi ^n).
\endaligned
\end{equation}
The first three terms on the right-hand side of \eqref{semi-discrete energy stability3} can be estimated with the help of the integration by parts:
\begin{equation}\label{semi-discrete energy stability4}
\aligned
(\Delta^2 \phi^{n+1/2}, \phi^{n+1}-\phi^n) =\frac1 2(\|\Delta \phi^{n+1}\|^2-\|\Delta \phi^{n}\|^2),
\endaligned
\end{equation}

\begin{equation}\label{semi-discrete energy stability5}
\aligned
&2(\Delta \tilde{ \phi }^{n+1/2}, \phi^{n+1}- \phi^n) \\
=&-\|\nabla \phi ^{n+1}\|^2+\|\nabla \phi^{n}\|^2+\frac1 2(\|\nabla \phi^{n+1}-\nabla \phi^{n}\|^2-\|\nabla \phi^{n}-\nabla \phi^{n-1}\|^2)\\
&+\frac1 2\|\nabla \phi^{n+1}-2\nabla \phi^{n}+\nabla \phi^{n-1}\|^2,
\endaligned
\end{equation}
and 
\begin{equation}\label{semi-discrete energy stability5_add}
\aligned
\alpha(\phi^{n+1/2}, \phi^{n+1}-\phi^n) =\frac\alpha 2(\| \phi^{n+1}\| ^2-\| \phi^{n}\| ^2).
\endaligned
\end{equation}
Multiplying \eqref{e_semi_second-order4} by $(r^{n+1}+r^n)$ leads to
\begin{equation}\label{semi-discrete energy stability6}
\aligned
(r^{n+1})^2-(r^n)^2=(\frac{r^{n+1/2}}{\sqrt{E_1 (\tilde{ \phi }^{n+1/2})}}F^{\prime}(\tilde{ \phi }^{n+1/2} ), \phi^{n+1}-\phi^n).
\endaligned
\end{equation}
Combining \eqref{semi-discrete energy stability3} with \eqref{semi-discrete energy stability2} and  \eqref{semi-discrete energy stability4}-\eqref{semi-discrete energy stability6}, we have 
\begin{equation}\label{semi-discrete energy stability7}
\aligned
&\frac1 2(\|\Delta \phi^{n+1}\|^2-\|\Delta \phi^{n}\| ^2)
-\|\nabla \phi^{n+1}\|^2+\|\nabla \phi^{n}\|^2+(r^{n+1})^2-(r^n)^2\\
+&\frac1 2(\|\nabla \phi^{n+1}-\nabla \phi^{n}\|^2-\|\nabla \phi^{n}-\nabla \phi^{n-1}\|^2)\\
+&\frac1 2\|\nabla \phi^{n+1}-2\nabla \phi^{n}+\nabla \phi^{n-1}\|^2+\frac\alpha 2(\|\phi^{n+1}\|^2-\| \phi^{n}\| ^2)\\
=&\Delta t(\psi^{n+1/2}, \mu^{n+1/2}).
\endaligned
\end{equation}
Since recalling \eqref{e_semi_second-order1}, we can derive 
\begin{equation}\label{semi-discrete energy stability8}
\aligned
\frac{1}{2M}(\|\psi^{n+1}\|_{H^{-1}}^2-&\|\psi^n\|_{H^{-1}}^2)=\frac 1 M(\psi^{n+1}-\psi^n,\psi^{n+1/2})_{-1}\\
=&-\frac 1 M(\psi^{n+1}-\psi^n,\Delta^{-1}\psi^{n+1/2})\\
%=&\beta\Delta t(\Psi^{n+1},\Delta_h^{-1}\Psi^{n+1})-M\Delta t(\Delta_hW^{n+1},\Delta_h^{-1}\Psi^{n+1})\\
=&-\frac{\beta}{M}\Delta t\|\psi^{n+1/2}\|_{H^{-1}}^2-\Delta t( \mu^{n+1/2},\psi^{n+1/2}).
\endaligned
\end{equation}
Finally, combining \eqref{semi-discrete energy stability7} with \eqref{semi-discrete energy stability8} gives the desired result. 
\end{proof}

Since the scheme (\ref{e_semi_second-order1})-(\ref{e_semi_second-order4}) is linear, one can also  show that it admits a unique solution, and can be efficiently implemented. For the sake of brevity, we shall provide detail only for the fully discretized scheme presented in the next section.

%==================================================================
%  \section{The Time Discretizations}
% It is very critical to consider the time discretization first when constructing a fully discrete energy-decaying scheme. In this section, we first construct the linearly first and second-order semi-implicit temporal discretization for the MPFC equation, and prove that semi-discrete schemes are mass conserving and unconditionally energy stable. 
    
%==================================================================
 \section{Fully discrete schemes and their properties}  
 In this section, we construct two linear SAV block-centered finite difference schemes  for the SAV reformulated  MPFC equation \eqref{e_model_recast}. 
 
 \subsection{Full discrete schemes based on block-centered finite difference method} 
First we describe briefly  the block-centered finite difference framework that we will employ to define and analyze our schemes. To fix the idea, we set  $\Omega=(0,L_x)\times 
(0,L_y)$, although the algorithm and analysis presented below apply also to the one- and three-dimensional rectangular domains.

We begin with the definitions of grid points and difference operators. Let $L_x=N_xh_x$ and $L_y=N_yh_y$, where $h_x$ and $h_y$ are grid spacings in $x$ and $y$ directions, and $N_x$ and 
$N_y$ are the number of grids along the $x$ and $y$ coordinates, respectively. The grid points are denoted by
$$(x_{i+1/2},y_{j+1/2}),\ \ i=0,...,N_x,\ \ j=0,...,N_y,$$
and 
\begin{equation*}
\aligned
&x_{i}=(x_{i-\frac{1}{2}}+x_{i+\frac{1}{2}})/2, \ \ i=1,...,N_x,\\
&y_{j}=(y_{j-\frac{1}{2}}+y_{j+\frac{1}{2}})/2, \ \ j=1,...,N_y.
\endaligned
\end{equation*}
Define
\begin{equation*}
\aligned
&[d_{x}g]_{i+\frac{1}{2},j}=(g_{i+1,j}-g_{i,j})/h_x,\\
&[d_{y}g]_{i,j+\frac{1}{2}}=(g_{i,j+1}-g_{i,j})/h_y,\\
&[D_{x}g]_{i,j}=(g_{i+\frac{1}{2},j}-g_{i-\frac{1}{2},j})/h_x,\\
&[D_{y}g]_{i,j}=(g_{i,j+\frac{1}{2}}-g_{i,j-\frac{1}{2}})/h_y,\\
&[\Delta_hg]_{i,j}=D_x(d_xg)_{i,j}+D_y(d_yg)_{i,j}.
\endaligned
\end{equation*}
Define the discrete inner products and norms as follows,
\begin{equation*}
\aligned
&(f,g)_{m}=\sum\limits_{i=1}^{N_{x}}\sum\limits_{j=1}^{N_{y}}h_xh_yf_{i,j}g_{i,j},\\
&(f,g)_{x}=\sum\limits_{i=1}^{N_{x}-1}\sum\limits_{j=1}^{N_{y}}h_xh_yf_{i+\frac{1}{2},j}g_{i+\frac{1}{2},j},\\
&(f,g)_{y}=\sum\limits_{i=1}^{N_{x}}\sum\limits_{j=1}^{N_{y}-1}h_xh_yf_{i,j+\frac{1}{2}}g_{i,j+\frac{1}{2}}.
\endaligned
\end{equation*}
\begin{lemma}\label{le_discrete-integration-by-part}
 Let $q_{i,j},w_{1,i+1/2,j}\ and \  w_{2,i,j+1/2} $ be any values such that $w_{1,1/2,j}=w_{1,N_x+1/2,j}=w_{2,i,1/2}=w_{2,i,N_y+1/2}=0$, then
$$(q,D_xw_1)_m=-(d_xq,w_1)_x,$$
$$(q,D_yw_2)_m=-(d_yq,w_2)_y.$$
\end{lemma} 
% ********************************************************************************************
% For periodic boundary condition
\begin{comment}
\begin{equation*}
\aligned
&(f,g)_{m}=\sum\limits_{i=1}^{N_{x}}\sum\limits_{j=1}^{N_{y}}h_xh_yf_{i,j}g_{i,j},\\
&(f,g)_{x}=\frac 1 2\sum\limits_{i=1}^{N_{x}}\sum\limits_{j=1}^{N_{y}}(f_{i+\frac{1}{2},j}g_{i+\frac{1}{2},j}+f_{i-\frac{1}{2},j}g_{i-\frac{1}{2},j})h_xh_y,\\
&(f,g)_{y}=\frac 1 2\sum\limits_{i=1}^{N_{x}}\sum\limits_{j=1}^{N_{y}}(f_{i,j+\frac{1}{2}}g_{i,j+\frac{1}{2}}+f_{i,j-\frac{1}{2}}g_{i,j-\frac{1}{2}} )h_xh_y.
\endaligned
\end{equation*}
The following discrete-integration-by-part plays an important role in the analysis.
\begin{lemma}\label{le_discrete-integration-by-part}
 Let $q_{i,j},w_{1,i+1/2,j} \ and \ w_{2,i,j+1/2} $ be any values, then we have
 \begin{equation*}
\aligned
&(d_xq,w_1)_x=-(q,D_xw_1)_m+\sum\limits_{j=1}^{N_y}(\frac{q_{N_x+1,j}+q_{N_x,j}}{2}w_{1,N_x+1/2,j}-\frac{q_{1,j}+q_{0,j}}{2}w_{1,1/2,j})h_y,\\
&(d_yq,w_2)_y=-(q,D_yw_2)_m+\sum\limits_{i=1}^{N_x}(\frac{q_{i,N_y+1}+q_{i,N_y}}{2}w_{2,i,N_y+1/2}-\frac{q_{i,1}+q_{i,0}}{2}w_{2,i,1/2})h_x.
\endaligned
\end{equation*}
\end{lemma}
\end{comment} 
% ********************************************************************************************
\medskip
Next we define the discrete $H^{-1}$ inner-product. Suppose $\eta_{\phi_i}\in \{f|(f,1)_m=0\}:=\mathcal{H}$ to be the unique solution to the following problem:
\begin{equation}\label{e_discrete_inner-product1}
\aligned
-\Delta_h\eta_{\phi_i}=\phi_i,
\endaligned
\end{equation} 
where $\eta_{\phi_i}$ satisfies the discrete homogenous Neumann boundary condition
 \begin{equation}\label{e_first-order10}
  \left\{
   \begin{array}{l}
    \displaystyle (\eta_{\phi_i})_{0,j}=(\eta_{\phi_i})_{1,j}, \ (\eta_{\phi_i})_{N_x+1,j}=(\eta_{\phi_i})_{N_x,j}, \ \ j=1,2,\ldots,N_y,\\
    \displaystyle (\eta_{\phi_i})_{k,0}=(\eta_{\phi_i})_{k,1}, \ \ (\eta_{\phi_i})_{k,N_y+1}=(\eta_{\phi_i})_{k,N_y}, \ \ k=1,2,\ldots,N_x.
   \end{array}
   \right.
  \end{equation}    
We define the bilinear form 
$$(\phi_1,\phi_2)_{-1}=(d_x\eta_{\phi_1},d_x\eta_{\phi_2})_x+(d_y\eta_{\phi_1},d_y\eta_{\phi_2})_y,$$
for any $\phi_1,\phi_2\in \mathcal{H}$.
Then we can obtain that $(\phi_1,\phi_2)_{-1}$ is an inner product on the space $\mathcal{H}$. Moreover, we have 
$$(\phi_1,\phi_2)_{-1}=-(\phi_1,\Delta^{-1}_h\phi_2)_m=-(\Delta^{-1}_h\phi_1,\phi_2)_m.$$
Then we can define the discrete $H^{-1}$ norm $\|\phi\|_{-1}=\sqrt{(\phi,\phi)_{-1}}$.

Hereafter, we use $C$, with or without subscript, to denote a positive
constant, which could have different values at different appearances.

%================================================================== 
% \subsection{The first-order scheme}
  %================================================================== 
 %\subsection{The second-order scheme}
 Let us denote by $\{Z^n, W^n, R^n,\Psi^n\}_{n=0}^{N}$ the block-centered finite difference approximations to $\{\phi^n,\mu^n, r^n,\psi^n\}_{n=0}^{N}$.  The second-order scheme defined by the Crank-Nicolsion method for  \eqref{e_model_recast} is as follows:
 
  Set the boundary condition as 
 \begin{equation}\label{e_first-order1}
  \left\{
   \begin{array}{l}
    \displaystyle Z_{0,j}=Z_{1,j}, \ Z_{N_x+1,j}=Z_{N_x,j}, \ \ j=1,2,\ldots,N_y,\\
    \displaystyle Z_{i,0}=Z_{i,1}, \ \ Z_{i,N_y+1}=Z_{i,N_y}, \ \ i=1,2,\ldots,N_x, \\
    \displaystyle W_{0,j}=W_{1,j}, \ W_{N_x+1,j}=W_{N_x,j}, \ \ j=1,2,\ldots,N_y,\\
    \displaystyle W_{i,0}=W_{i,1}, \ \ W_{i,N_y+1}=W_{i,N_y}, \ \ i=1,2,\ldots,N_x, \\
        \displaystyle \Delta_h Z_{0,j}=\Delta_h Z_{1,j}, \ \Delta_h Z_{N_x+1,j}=\Delta_h Z_{N_x,j}, \ \ j=1,2,\ldots,N_y,\\
    \displaystyle \Delta_h Z_{i,0}=\Delta_h Z_{i,1}, \ \ \Delta_h Z_{i,N_y+1}=\Delta_h Z_{i,N_y}, \ \ i=1,2,\ldots,N_x.
   \end{array}
   \right.
  \end{equation}
 We find $\{Z^{n+1}, W^{n+1}, R^{n+1},\Psi^{n+1}\}_{n=0}^{N-1}$ such that
 \begin{numcases}{}
 \Psi^{n+1}-\Psi^{n}+\beta\Delta t\Psi^{n+1/2}=M\Delta t\Delta_hW^{n+1/2}, \label{e_second-order1}\\
 \Delta t\Psi^{n+1/2}=Z^{n+1}-Z^n, \label{e_second-order2}\\
 W^{n+1/2}=\Delta_h^2Z^{n+1/2}+2\Delta_h\tilde{Z}^{n+1/2}+\alpha Z^{n+1/2}\notag\\
 \ \ \ \ \ +\frac{R^{n+1/2}}{\sqrt{E_1^h(\tilde{Z}^{n+1/2})}}F^{\prime}(\tilde{Z}^{n+1/2}),  \label{e_second-order3}\\
R^{n+1}-R^n=\frac{1}{2\sqrt{E_1^h(\tilde{Z}^{n+1/2}) }}(F^{\prime}(\tilde{Z}^{n+1/2}),
Z^{n+1}-Z^n)_m,  \label{e_second-order4}
\end{numcases}
where $f^{n+1/2}=(f^{n+1}+f^{n})/2,\ f=W,\Psi, R$ and  $\tilde{Z}^{n+1/2}=(3Z^n-Z^{n-1})/2$. For the case of $n=0$, we can computer $\tilde{Z}^{1/2}$ by the first-order scheme.
%================================================================== 
\subsection{Efficient implementation}\label{eff_impl}
 
A  remarkable property about the above schemes is that it can be solved very efficiently.
We demonstrate the detail procedure to solve the second-order SAV scheme (\ref{e_second-order1})-(\ref{e_second-order4}). Indeed, we can eliminate $\Psi^{n+1}$, $W^{n+1}$, $R^{n+1}$ from (\ref{e_second-order1})-(\ref{e_second-order4}) to obtain
\begin{equation}\label{e_implementation1}
\aligned
&(\frac{2}{\Delta t}+\beta)\frac{Z^{n+1}-Z^n}{\Delta t}-\frac{2 }{\Delta t}  \Psi^{n}=M( \frac{1}{2} \Delta_h^3Z^{n+1}+ \frac{1}{2} \Delta_h^3Z^{n} + 2\Delta_h^2 \tilde{Z}^{n+1/2} \\
&\ \ \ \ \ \ + \frac{\alpha}{2} \Delta_hZ^{n+1} + \frac{\alpha}{2} \Delta_hZ^{n} ) 
 +M \frac{\Delta_hF^{\prime}(\tilde{Z}^{n+1/2} )}{ \sqrt{E_1^h(\tilde{Z}^{n+1/2} )}}\left(R^n+ (\frac{F^{\prime}(\tilde{Z}^{n+1/2} )}{4\sqrt{E_1^h(\tilde{Z}^{n+1/2} )}},Z^{n+1}-Z^n)_m\right).
\endaligned
\end{equation} 
Let $b^n=\frac{F^{\prime}(\tilde{Z}^{n+1/2} )}{\sqrt{E_1^h(\tilde{Z}^{n+1/2} )}}$, then the above equation can be transformed into the following:
\begin{equation}\label{e_implementation2}
\aligned
&\mathcal{A}Z^{n+1}-\frac{M}{4}(b^n,Z^{n+1})_m\Delta_h b^n=f^n,
\endaligned
\end{equation}
where $\mathcal{A}=(\frac{2 }{\Delta t^2}+\frac{\beta}{\Delta t})I-\frac{M}{2} \Delta_h^3-\frac{M}{2} \alpha\Delta_h$ and 
the right term 
\begin{equation*}
\aligned
f^n=& \frac{2}{\Delta t}\Psi^n+\left( (\frac{2}{\Delta t^2}+\frac{\beta}{\Delta t})I+ \frac{M}{2} \Delta_h^3 + \frac{M}{2} \alpha\Delta_h \right)Z^n \\
& +2M\Delta_h^2 \tilde{Z}^{n+1/2} +M\left(R^n-\frac 1 4(b^n,Z^n)_m\right)\Delta_hb^n.
\endaligned
\end{equation*}
In order to solve the above equation, we should determine $(b^n,Z^{n+1})_m$ first. To this end, multiplying \eqref{e_implementation2} by $\mathcal{A}^{-1}$ leads to
\begin{equation}\label{e_implementation3}
\aligned
&Z^{n+1}-\frac{M}{4}(b^n,Z^{n+1})_m\mathcal{A}^{-1}\Delta_h b^n=\mathcal{A}^{-1}f^n.
\endaligned
\end{equation}
Multiplying \eqref{e_implementation3} by $b^n_{i,j}h_xh_y$, and making summation on $i,j$ for $1\leq i\leq N_x,\ 1\leq j\leq N_y$, we have
%\begin{equation}\label{e_implementation4}
%\aligned
%&\left(1-\frac{M}{2}(\mathcal{A}^{-1}\Delta_h b^n,b^n)_m\right)(b^n,Z^{n+1})_m=(b^n,\mathcal{A}^{-1}f^n)_m.
%\endaligned
%\end{equation}
%Then we have
\begin{equation}\label{e_implementation5}
\aligned
(b^n,Z^{n+1})_m=\frac{(b^n,\mathcal{A}^{-1}f^n)_m}{1-\frac{M}{4 }(\mathcal{A}^{-1}\Delta_h b^n,b^n)_m}.
\endaligned
\end{equation}
Since $M>0$ and for $\alpha,\,\beta\ge 0$, $\mathcal{A}^{-1}\Delta_h$ is negative definite. So 
$(b^n,Z^{n+1})_m$ can be uniquely determined from above. Finally, we can get $Z^{n+1}$ by \eqref{e_implementation3}. Since the scheme is linear, the above procedure shows that it admits a unique solution.

In conclusion, the second-order SAV scheme (\ref{e_second-order1})-(\ref{e_second-order4}) can be effectively implemented in the following algorithm:
\begin{algorithm}
%\setcounter{algorithm}{2}
%\caption{Effective SAV procedure}
\ \textbf{Given:} $\Psi^n$, $Z^n$, $R^n$ and $b^n$.
\begin{algorithmic}
\STATE Step 1. Computer $(\mathcal{A}^{-1}\Delta_h b^n,b^n)_m$. This can be accomplished by 
solving a sixth-order equation with constant coefficients.\\
\STATE Step 2. Calculate $(b^n,Z^{n+1})_m$ using \eqref{e_implementation5}, which requires solving another sixth-order equation $\mathcal{A}^{-1}f^n$ with constant coefficients.\\
\STATE Step 3. Update $Z^{n+1}$ by
$Z^{n+1}=\frac{M}{4 }(b^n,Z^{n+1})_m\mathcal{A}^{-1}\Delta_h b^n+\mathcal{A}^{-1}f^n.
$
\end{algorithmic}
\end{algorithm}

While the second-order scheme above is suitable in most situations, there are cases, e.g., when only steady state solutions are desired,  where a first-order scheme is preferred. For the readers' convenience, we list 
 the first-order SAV scheme below:
 
We find $\{Z^{n+1}, W^{n+1}, R^{n+1},\Psi^{n+1}\}_{n=0}^{N-1}$ such that
 \begin{numcases}{}
 \Psi^{n+1}-\Psi^{n}+\beta\Delta t\Psi^{n+1}=M\Delta t\Delta_hW^{n+1}, \label{e_first-order2}\\
 \Delta t\Psi^{n+1}=Z^{n+1}-Z^n, \label{e_first-order3}\\
 W^{n+1}=\Delta_h^2Z^{n+1}+2\Delta_hZ^{n}+\alpha Z^{n+1}+\frac{R^{n+1}}{\sqrt{E_1^h(Z^n)}}F^{\prime}(Z^n),  \label{e_first-order4}\\
R^{n+1}-R^n=\frac{1}{2\sqrt{E_1^h(Z^{n})}}(F^{\prime}(Z^{n}),
Z^{n+1}-Z^n)_m,  \label{e_first-order5}
\end{numcases}
where the discrete form of $E_1(Z^{n})$ is defined as follows:
$$E_1^h(Z^{n})=\sum\limits_{i=1}^{N_{x}}\sum\limits_{j=1}^{N_{y}}h_xh_yF(Z^{n}_{i,j}).$$

%==================================================================
 \subsection{Mass conservation and unconditional energy stability}  
% In this section, we demonstrate the proofs of the mass conservation and unconditional energy stability for the  second-order  SAV scheme (\ref{e_second-order1})-(\ref{e_second-order4}).
 
Define the discrete pseudo energy 
\begin{equation}\label{e_discrete pseudo energy}
\aligned
\mathcal{E}_d(Z^n,R^n,\Psi^n)=\frac 1 2 \|\Delta_hZ^n\|_m^2-\|\nabla_hZ^n\|^2+\frac \alpha 2 \|Z^{n+1}\|_m^2+R^2+\frac{1}{2M}\|\Psi^n\|_{H^{-1}}^2,
\endaligned
\end{equation} 
where $\|\nabla_hZ\|=\sqrt{(d_xZ,d_xZ)_x+(d_yZ,d_yZ)_y}$.

\begin{theorem}\label{thm: second-order mass conservation}
The scheme (\ref{e_second-order1})-(\ref{e_second-order4}) admits a unique solution, is mass conserving, i.e., $(Z^{n+1},1)_m=(Z^{n},1)_m$ for all $n$, and 
unconditionally stable in the sense that
\begin{equation}\label{second-order energy stability1}
\aligned
\tilde{\mathcal{E}}_d(Z^{n+1},R^{n+1},\Psi^{n+1})-\tilde{\mathcal{E}}_d(Z^n,R^n,\Psi^n)\leq -\frac{\beta}{M}\Delta t\|\Psi^{n+1/2}\|_{H^{-1}}^2,
\endaligned
\end{equation}
where $\tilde{\mathcal{E}}_d(Z^n,R^n,\Psi^n)=\mathcal{E}_d(Z^n,R^n,\Psi^n)+\frac 1 2\|\nabla_hZ^{n}-\nabla_hZ^{n-1}\|^2$.
\end{theorem}

\begin{proof}
Since the scheme (\ref{e_second-order1})-(\ref{e_second-order4}) is linear, the algorithm describes in Subsection \ref{eff_impl} indicates that it admits a unique solution. The proof for mass conservation and energy dissipation is essentially the same as that for the semi-discrete case. For the readers' convenience, we still provide details below.

Summing \eqref{e_second-order1} on $i,j$ for $1\leq i\leq N_x,~1\leq j\leq N_y$ leads to 
\begin{equation}\label{second-order mass conservation1}
\aligned
(\Psi^{n+1}-\Psi^{n},1)_m+\beta\Delta t(\Psi^{n+1/2},1)_m=M\Delta t(\Delta_hW^{n+1/2},1)_m.
\endaligned
\end{equation}
Similarly, by summing \eqref{e_second-order2}, we can obtain
\begin{equation}\label{second-order mass conservation2}
\aligned
(Z^{n+1}-Z^n,1)_m=\Delta t(\Psi^{n+1/2},1)_m.
\endaligned
\end{equation}
Taking notice of Lemma \ref{le_discrete-integration-by-part} and the boundary condition \eqref{e_first-order1}, the term on the right hand side of \eqref{second-order mass conservation1} can be transformed into
\begin{equation}\label{second-order mass conservation3}
\aligned
M\Delta t(\Delta_hW^{n+1/2},1)_m=-M\Delta t\left((d_xW^{n+1/2},d_x1)_x+(d_yW^{n+1/2},d_y1)_y\right)=0.
\endaligned
\end{equation}
Then \eqref{second-order mass conservation1} can be estimated as follows:
\begin{equation}\label{second-order mass conservation4}
\aligned
(1+\frac \beta 2\Delta t)(\Psi^{n+1},1)_m=(1-\frac \beta 2\Delta t)(\Psi^{n},1)_m.
\endaligned
\end{equation}
Combining \eqref{second-order mass conservation4} with the condition on the initial condition $(\Psi^0,1)=0$ leads to $(\Psi^{n+1},1)_m=0$ for all $n\leq 0$. Recalling \eqref{second-order mass conservation2}, we have 
$(Z^{n+1},1)_m=(Z^{n},1)_m$.

Next, we prove \eqref{second-order energy stability1}.
Multiplying \eqref{e_second-order2} by $W_{i,j}^{n+1/2}h_xh_y$, and making summation on $i,j$ for $1\leq i\leq N_x,\ 1\leq j\leq N_y$, we have
\begin{equation}\label{second-order energy stability2}
\aligned
\Delta t(\Psi^{n+1/2},W^{n+1/2})_m=(Z^{n+1}-Z^n,W^{n+1/2})_m.
\endaligned
\end{equation}
Multiplying \eqref{e_second-order3} by $(Z^{n+1}_{i,j}-Z^n_{i,j})h_xh_y$, and making summation on $i,j$ for $1\leq i\leq N_x,\ 1\leq j\leq N_y$, we have
\begin{equation}\label{second-order energy stability3}
\aligned
&(W^{n+1/2},Z^{n+1}-Z^n)_m=(\Delta_h^2Z^{n+1/2},Z^{n+1}-Z^{n})_m+2(\Delta_h\tilde{Z}^{n+1/2},Z^{n+1}-Z^n)_m\\
&\ \ \ \ \ +\alpha(Z^{n+1/2},Z^{n+1}-Z^{n})_m+(\frac{R^{n+1/2}}{\sqrt{E_1^h(\tilde{Z}^{n+1/2})}}F^{\prime}(\tilde{Z}^{n+1/2}),Z^{n+1}-Z^n)_m.
\endaligned
\end{equation}
The first three terms on the right-hand side of \eqref{second-order energy stability3} can be dealt with the help of Lemma \ref{le_discrete-integration-by-part} and the boundary condition \eqref{e_first-order1}:
\begin{equation}\label{second-order energy stability4}
\aligned
(\Delta_h^2Z^{n+1/2},Z^{n+1}-Z^n)_m=\frac1 2(\|\Delta_hZ^{n+1}\|_m^2-\|\Delta_hZ^{n}\|_m^2).
\endaligned
\end{equation}

\begin{equation}\label{second-order energy stability5}
\aligned
&2(\Delta_h\tilde{Z}^{n+1/2},Z^{n+1}-Z^n)_m\\
=&-\|\nabla_hZ^{n+1}\|^2+\|\nabla_hZ^{n}\|^2+\frac1 2(\|\nabla_hZ^{n+1}-\nabla_hZ^{n}\|^2-\|\nabla_hZ^{n}-\nabla_hZ^{n-1}\|^2)\\
&+\frac1 2\|\nabla_hZ^{n+1}-2\nabla_hZ^{n}+\nabla_hZ^{n-1}\|^2.
\endaligned
\end{equation}

\begin{equation}\label{second-order energy stability5_add}
\aligned
\alpha(Z^{n+1/2},Z^{n+1}-Z^n)_m=\frac\alpha 2(\|Z^{n+1}\|_m^2-\|Z^{n}\|_m^2).
\endaligned
\end{equation}
Multiplying \eqref{e_second-order4} by $(R^{n+1}+R^n)$ leads to
\begin{equation}\label{second-order energy stability6}
\aligned
(R^{n+1})^2-(R^n)^2=(\frac{R^{n+1/2}}{\sqrt{E_1^h(\tilde{Z}^{n+1/2} )}}F^{\prime}( \tilde{Z}^{n+1/2} ),Z^{n+1}-Z^n)_m.
\endaligned
\end{equation}
Combining \eqref{second-order energy stability3} with \eqref{second-order energy stability2} and  \eqref{second-order energy stability4}-\eqref{second-order energy stability6}, we have 
\begin{equation}\label{second-order energy stability7}
\aligned
&\frac1 2(\|\Delta_hZ^{n+1}\|_m^2-\|\Delta_hZ^{n}\|_m^2)
-\|\nabla_hZ^{n+1}\|^2+\|\nabla_hZ^{n}\|^2+(R^{n+1})^2-(R^n)^2\\
+&\frac1 2(\|\nabla_hZ^{n+1}-\nabla_hZ^{n}\|^2-\|\nabla_hZ^{n}-\nabla_hZ^{n-1}\|^2)\\
+&\frac1 2\|\nabla_hZ^{n+1}-2\nabla_hZ^{n}+\nabla_hZ^{n-1}\|^2+\frac\alpha 2(\|Z^{n+1}\|_m^2-\|Z^{n}\|_m^2)\\
=&\Delta t(\Psi^{n+1/2},W^{n+1/2})_m.
\endaligned
\end{equation}
Since recalling \eqref{e_second-order1}, we can derive 
\begin{equation}\label{second-order energy stability8}
\aligned
\frac{1}{2M}(\|\Psi^{n+1}\|_{H^{-1}}^2-&\|\Psi^n\|_{H^{-1}}^2)=\frac 1 M(\Psi^{n+1}-\Psi^n,\Psi^{n+1/2})_{-1}\\
=&-\frac 1 M(\Psi^{n+1}-\Psi^n,\Delta_h^{-1}\Psi^{n+1/2})_m\\
%=&\beta\Delta t(\Psi^{n+1},\Delta_h^{-1}\Psi^{n+1})-M\Delta t(\Delta_hW^{n+1},\Delta_h^{-1}\Psi^{n+1})\\
=&-\frac{\beta}{M}\Delta t\|\Psi^{n+1/2}\|_{H^{-1}}^2-\Delta t(W^{n+1/2},\Psi^{n+1/2})_m.
\endaligned
\end{equation}
Finally, combining \eqref{second-order energy stability7} with \eqref{second-order energy stability8} gives the desired result. 
\end{proof}
%==================================================================
 \section{Error analysis}  
 In this section, we carry out a rigorous error analysis for the second-order scheme \eqref{e_second-order1}-\eqref{e_second-order4}. %Since the proofs for the first-order scheme \eqref{e_first-order2}-\eqref{e_first-order5} are essentially the same as for the second-order scheme, so we skip that for brevity. 
  
 Set
\begin{equation*}
\aligned
& e_{\phi}^n=Z^n-\phi^n,\  e_{\psi}^n=\Psi^n-\psi^n,\\
& e_{\mu}^{n}=W^{n}-\mu^n,\ e_{r}^n=R^n-r^n.
\endaligned
\end{equation*} 

We start by proving the following lemma  which will be used to control the backward diffusion term in the error analysis.
\begin{lemma}\label{le: control of backward diffusion term}
 Suppose that $\phi$ and  $\Delta_h\phi$ are satisfy the homogeneous Neumann boundary conditions, then we have
\begin{equation}\label{second-order error estimate1}
\aligned
\|\Delta_h\phi\|_m^2\leq \frac{1}{3\epsilon^2}\|\phi\|_m^2+\frac{2\epsilon}{3}\|\nabla_h(\Delta_h\phi)\|^2.
\endaligned
\end{equation} 
\end{lemma}
\begin{proof}
The proof for the homogeneous Neumann boundary condition is essentially the same as for the periodic boundary condition. One can refer to \cite[Lemma 3.10]{wise2009energy} for more detail.
\end{proof}

\medskip

\begin{theorem}\label{thm: second-order error estimate} 
We assume that $\phi\in W^{4,\infty}(J;L^{\infty}(\Omega)) \cap L^{\infty}(J;W^{6,\infty}(\Omega))\cap W^{2,\infty}(J;W^{4,\infty}(\Omega))$. Let $\Delta t\leq C(h_x+h_y)$, then for the discrete scheme \eqref{e_second-order1}-\eqref{e_second-order4}, there exists a positive constant $C$ independent of $h_x$, $h_y$ and $\Delta t$ such that
\begin{equation}\label{second-order error estimate2}
\aligned
&\|Z^{k+1}-\phi^{k+1}\|_m+\|\nabla_h(\Delta_hZ^{k+1})-\nabla_h(\Delta_h\phi^{k+1})\|\\
&+\|\Delta_hZ^{k+1}-\Delta_h\phi^{k+1}\|_m+
|R^{k+1}-r^{k+1}|\\
\leq&C(\|\phi\|_{W^{4,\infty}(J;L^{\infty}(\Omega))}+\|\phi\|_{W^{2,\infty}(J;W^{4,\infty}(\Omega))})\Delta t^2\\
&+C\|\phi\|_{L^{\infty}(J;W^{8,\infty}(\Omega))}(h_x^2+h_y^2),\quad \forall 0\le k\le N-1.
\endaligned
\end{equation}
\end{theorem}

\begin{proof}
Subtracting equation \eqref{e_model_recastA} from equation \eqref{e_second-order1}, we obtain
\begin{equation}\label{second-order error estimate3}
\aligned
\frac{e_{\psi}^{n+1}-e_{\psi}^{n}}{\Delta t}+\beta e_{\psi}^{n+1/2}=M\Delta_he_{\mu}^{n+1/2}+T_1^{n+1/2},
\endaligned
\end{equation}
where 
\begin{equation}\label{second-order error estimate4}
\aligned
T_{1}^{n+1/2}=\frac{\partial \psi}{\partial t}\big|_{t=n+1/2}-\frac{\psi^{n+1}-\psi^{n}}{\Delta t}
\leq
C\|\psi\|_{W^{3,\infty}(J;L^{\infty}(\Omega))}\Delta t^2.
\endaligned
\end{equation}
Recalling \eqref{e_auxiliary1} and \eqref{e_second-order2}, we have
\begin{equation}\label{second-order error estimate5}
\aligned
e_{\psi}^{n+1/2}=\frac{e_{\phi}^{n+1}-e_{\phi}^n}{\Delta t}+T_2^{n+1/2},
\endaligned
\end{equation}
where
\begin{equation}\label{second-order error estimate6}
\aligned
T_{2}^{n+1/2}=\frac{\phi^{n+1}-\phi^{n}}{\Delta t}-\frac{\partial \phi}{\partial t}\big|_{t=n+1/2}
\leq C\|\phi\|_{W^{3,\infty}(J;L^{\infty}(\Omega))}\Delta t^2.
\endaligned
\end{equation}
Subtracting \eqref{e_model_recastB} from \eqref{e_second-order3} leads to
\begin{equation}\label{second-order error estimate7}
\aligned
e_{\mu}^{n+1/2}=&\Delta_h^2e_{\phi}^{n+1/2}+2\Delta_h\tilde{e}_{\phi}^{n+1/2}+\alpha 
e_{\phi}^{n+1/2}+\frac{R^{n+1/2}}{\sqrt{E_1^h(\tilde{Z}^{n+1/2})}}F^{\prime}(\tilde{Z}^{n+1/2})\\
&-\frac{r^{n+1/2}}{\sqrt{E_1(\phi^{n+1/2})}}F^{\prime}(\phi^{n+1/2})+T_3^{n+1/2},
\endaligned
\end{equation}
where
\begin{equation}\label{second-order error estimate8}
\aligned
T_{3}^{n+1/2}=&\Delta_h^2\phi^{n+1/2}-\Delta^2\phi^{n+1/2}+2\Delta_h\tilde{\phi}^{n+1/2}-2\Delta\phi^{n+1/2}\\
\leq	&C(\|\phi\|_{L^{\infty}(J;W^{6,\infty}(\Omega))}+\|\phi\|_{L^{\infty}(J;W^{4,\infty}(\Omega))})(h_x^2+h_y^2)\\
&+C\|\phi\|_{W^{2,\infty}(J;W^{2,\infty}(\Omega))}\Delta t^2.
\endaligned
\end{equation}
Subtracting \eqref{e_model_recastC} from \eqref{e_second-order4} gives that
\begin{equation}\label{second-order error estimate9}
\aligned
\frac{e_r^{n+1}-e_r^{n}}{\Delta t}=&\frac{1}{2\sqrt{E_1^h(\tilde{Z}^{n+1/2})}}(F^{\prime}(\tilde{Z}^{n+1/2}),
\frac{Z^{n+1}-Z^{n}}{\Delta t})_m\\
&-\frac{1}{2\sqrt{E_1(\phi^{n+1/2})}}\int_{\Omega}F^{\prime}(\phi^{n+1/2})\phi^{n+1/2}_t d\textbf{x}+T_{4}^{n+1/2},
\endaligned
\end{equation}
where 
\begin{equation}\label{second-order error estimate10}
\aligned
T_4^{n+1/2}=r_t^{n+1/2}-\frac{r^{n+1}-r^n}{\Delta t}\leq
C\|r\|_{W^{3,\infty}(J)}\Delta t^2.
\endaligned
\end{equation}
Multiplying \eqref{second-order error estimate3} by $e_{\psi,i,j}^{n+1/2}h_xh_y$, and making summation on $i,j$ for $1\leq i\leq N_x,~1\leq j\leq N_y$, we have 
\begin{equation}\label{second-order error estimate11}
\aligned
(\frac{e_{\psi}^{n+1}-e_{\psi}^{n}}{\Delta t},e_{\psi}^{n+1/2})_m+\beta\|e_{\psi}^{n+1/2}\|_m^2
=&M(\Delta_he_{\mu}^{n+1/2},e_{\psi}^{n+1/2})_m+(T_1^{n+1/2}, e_{\psi}^{n+1/2})_m.
\endaligned
\end{equation} 
The first term on the left-hand side of \eqref{second-order error estimate11} can be transformed into the following
\begin{equation}\label{second-order error estimate12}
\aligned
(\frac{e_{\psi}^{n+1}-e_{\psi}^{n}}{\Delta t},e_{\psi}^{n+1/2})_m=\frac{\|e_{\psi}^{n+1}\|_m^2-\|e_{\psi}^{n}\|_m^2}{2\Delta t}.
\endaligned
\end{equation} 
Taking notice of \eqref{second-order error estimate7}, we can write the first term on the right-hand side of \eqref{second-order error estimate11} as
\begin{equation}\label{second-order error estimate13}
\aligned
&M(\Delta_he_{\mu}^{n+1/2},e_{\psi}^{n+1/2})_m=M(\Delta_h^3e_{\phi}^{n+1/2},e_{\psi}^{n+1/2})_m+2M(\Delta_h^2\tilde{e}_{\phi}^{n+1/2},e_{\psi}^{n+1/2})_m\\
&\ \ \ \ \ +M(\frac{R^{n+1/2}}{\sqrt{E_1^h(\tilde{Z}^{n+1/2})}}\Delta_hF^{\prime}(\tilde{Z}^{n+1/2})-\frac{r^{n+1/2}}{\sqrt{E_1(\phi^{n+1/2})}}\Delta_hF^{\prime}(\phi^{n+1/2}),e_{\psi}^{n+1/2})_m\\
& \ \ \ \ \ +\alpha(\Delta_he_{\phi}^{n+1/2},e_{\psi}^{n+1/2})_m+M(\Delta_hT_3^{n+1/2},e_{\psi}^{n+1/2})_m.
\endaligned
\end{equation} 
Using Lemma \ref{le_discrete-integration-by-part} and the boundary condition \eqref{e_first-order1}, we can write the first and second terms on the right-hand side of \eqref{second-order error estimate13} as
\begin{equation}\label{second-order error estimate14}
\aligned
M(\Delta_h^3e_{\phi}^{n+1/2},e_{\psi}^{n+1/2})_m=&-M(\nabla_h(\Delta_he_{\phi}^{n+1/2}),\nabla_h(\Delta_he_{\psi}^{n+1/2}))\\
=&-M\frac{\|\nabla_h(\Delta_he_{\phi}^{n+1})\|^2-\|\nabla_h(\Delta_he_{\phi}^{n})\|^2
}{2\Delta t}.
\endaligned
\end{equation} 

\begin{equation}\label{second-order error estimate15}
\aligned
&2M(\Delta_h^2\tilde{e}_{\phi}^{n+1/2},e_{\psi}^{n+1/2})_m=M(\Delta_h(3e_{\phi}^n-e_{\phi}^{n-1}), \Delta_he_{\psi}^{n+1/2})_m\\
=&\frac{M}{\Delta t}\left(\|\Delta_he_{\phi}^{n+1}\|_m^2-\|\Delta_he_{\phi}^{n}\|_m^2-\frac{1}{2}(\|\Delta_he_{\phi}^{n+1}-\Delta_he_{\phi}^{n}\|_m^2-\|\Delta_he_{\phi}^{n}-\Delta_he_{\phi}^{n-1}\|_m^2)\right)\\
&-\frac{M}{2\Delta t}\|\Delta_he_{\phi}^{n+1}-2\Delta_he_{\phi}^{n}+\Delta_he_{\phi}^{n-1}\|_m^2.
\endaligned
\end{equation} 
The third term on the right-hand side of \eqref{second-order error estimate13} can be estimated by
\begin{equation}\label{second-order error estimate16}
\aligned
&M(\frac{R^{n+1/2}}{\sqrt{E_1^h(\tilde{Z}^{n+1/2})}}\Delta_hF^{\prime}(\tilde{Z}^{n+1/2})-\frac{r^{n+1/2}}{\sqrt{E_1(\phi^{n+1/2})}}\Delta_hF^{\prime}(\phi^{n+1/2}),e_{\psi}^{n+1/2})_m\\
=&Mr^{n+1/2}(\frac{\Delta_hF^{\prime}(\tilde{Z}^{n+1/2})}{\sqrt{E_1^h(\tilde{Z}^{n+1/2})}}-\frac{\Delta_hF^{\prime}(\tilde{\phi}^{n+1/2})}{\sqrt{E_1^h(\tilde{\phi}^{n+1/2})}}, e_{\psi}^{n+1/2})_m\\
&+Mr^{n+1/2}(\frac{\Delta_hF^{\prime}(\tilde{\phi}^{n+1/2})}{\sqrt{E_1^h(\tilde{\phi}^{n+1/2})}}-\frac{\Delta_hF^{\prime}(\phi^{n+1/2})}{\sqrt{E_1(\phi^{n+1/2})}}, e_{\psi}^{n+1/2})_m\\
&+Me_{r}^{n+1/2}(\frac{\Delta_hF^{\prime}(\tilde{Z}^{n+1/2})}{\sqrt{E_1^h(\tilde{Z}^{n+1/2})}},
e_{\psi}^{n+1/2})_m.
\endaligned
\end{equation} 

Below we shall first assume that there exist three positive constants $C_1$, $C_2$ and $C_3$ such that 
\begin{equation}\label{second-order error estimate17_hypotheses}
\aligned
\|Z^n\|_{{L^{\infty}(\Omega)}}\leq C_1, \ \ \|\nabla_hZ^n\|_{{L^{\infty}(\Omega)}}\leq C_2, \ \|\Delta_hZ^n\|_{{L^{\infty}(\Omega)}}\leq C_3,\quad\forall 0\le n\le N,
\endaligned
\end{equation}
which will be verified late in the proof.

Applying Lemma \ref{le: control of backward diffusion term}, the first term on the right-hand side of \eqref{second-order error estimate16} can be controlled similar to the estimates in \cite{wang2011energy} by 
\begin{equation}\label{second-order error estimate18}
\aligned
&Mr^{n+1/2}(\frac{\Delta_hF^{\prime}(\tilde{Z}^{n+1/2})}{\sqrt{E_1^h(\tilde{Z}^{n+1/2})}}-\frac{\Delta_hF^{\prime}(\tilde{\phi}^{n+1/2})}{\sqrt{E_1^h(\tilde{\phi}^{n+1/2})}}, e_{\psi}^{n+1/2})_m\\
\leq &C(\|e_{\phi}^n\|_m^2+\|\Delta_he_{\phi}^n\|_m^2)+C(\|e_{\phi}^{n-1}\|_m^2+\|\Delta_he_{\phi}^{n-1}\|_m^2)+C\|e_{\psi}^{n+1/2}\|_m^2\\
\leq&C(\|e_{\phi}^n\|_m^2+\|\nabla_h(\Delta_he_{\phi}^n)\|^2)+C(\|e_{\phi}^{n-1}\|_m^2+\|\nabla_h(\Delta_he_{\phi}^{n-1})\|^2)+C\|e_{\psi}^{n+1/2}\|_m^2,
\endaligned
\end{equation}
where $C$ is dependent on $\|r\|_{L^{\infty}(J)},\ \|Z^n\|_{L^{\infty}(\Omega)}, \|\nabla_hZ^n\|_{L^{\infty}(\Omega)}$.

The second term on the right-hand side of \eqref{second-order error estimate16} can be handled by:
\begin{equation}\label{second-order error estimate19}
\aligned
&Mr^{n+1/2}(\frac{\Delta_hF^{\prime}(\tilde{\phi}^{n+1/2})}{\sqrt{E_1^h(\tilde{\phi}^{n+1/2})}}-\frac{\Delta_hF^{\prime}(\phi^{n+1/2})}{\sqrt{E_1(\phi^{n+1/2})}}, e_{\psi}^{n+1/2})_m\\
=&Mr^{n+1/2}(\frac{\Delta_hF^{\prime}(\tilde{\phi}^{n+1/2})}{\sqrt{E_1^h(\tilde{\phi}^{n+1/2})}}-\frac{\Delta_hF^{\prime}(\phi^{n+1/2})}{\sqrt{E_1^h(\tilde{\phi}^{n+1/2})}}, e_{\psi}^{n+1/2})_m\\
&+Mr^{n+1/2}(\frac{\Delta_hF^{\prime}(\phi^{n+1/2})}{\sqrt{E_1^h(\tilde{\phi}^{n+1/2})}}-\frac{\Delta_hF^{\prime}(\phi^{n+1/2})}{\sqrt{E_1(\phi^{n+1/2})}}, e_{\psi}^{n+1/2})_m\\
\leq	&C\|e_{\psi}^{n+1/2}\|_m^2+C\|\phi\|_{W^{2,\infty}(J;W^{2,\infty}(\Omega))}\Delta t^4\\
&+C\|\phi\|_{L^{\infty}(J;W^{3,\infty}(\Omega))}(h_x^4+h_y^4).
\endaligned
\end{equation} 
 The last term on the right-hand side of \eqref{second-order error estimate16} can be directly controlled by Cauchy-Schwarz inequality:
 \begin{equation}\label{second-order error estimate20}
\aligned
Me_{r}^{n+1/2}(\frac{\Delta_hF^{\prime}(\tilde{Z}^{n+1/2})}{\sqrt{E_1^h(\tilde{Z}^{n+1/2})}},
e_{\psi}^{n+1/2})_m\leq C\|e_{\psi}^{n+1/2}\|_m^2+C(e_r^{n+1/2})^2,
\endaligned
\end{equation}
where $C$ is dependent on $\|Z^n\|_{L^{\infty}(\Omega)},\  \|\nabla_hZ^n\|_{L^{\infty}(\Omega)}, \ \|\Delta_hZ^n\|_{L^{\infty}(\Omega)}$. 
Applying estimates \eqref{second-order error estimate18}-\eqref{second-order error estimate20} yields
\begin{equation}\label{second-order error estimate21}
\aligned
&M(\frac{R^{n+1/2}}{\sqrt{E_1^h(\tilde{Z}^{n+1/2})}}\Delta_hF^{\prime}(\tilde{Z}^{n+1/2})-\frac{r^{n+1/2}}{\sqrt{E_1(\phi^{n+1/2})}}\Delta_hF^{\prime}(\phi^{n+1/2}),e_{\psi}^{n+1/2})_m\\
\leq&C(\|e_{\phi}^n\|_m^2+\|\nabla_h(\Delta_he_{\phi}^n)\|^2)+C(\|e_{\phi}^{n-1}\|_m^2+\|\nabla_h(\Delta_he_{\phi}^{n-1})\|^2)\\
&+C\|e_{\psi}^{n+1/2}\|_m^2+C(e_r^{n+1/2})^2+C\|\phi\|^2_{W^{2,\infty}(J;W^{2,\infty}(\Omega))}\Delta t^4\\
&+C\|\phi\|^2_{L^{\infty}(J;W^{3,\infty}(\Omega))}(h_x^4+h_y^4).
\endaligned
\end{equation} 
The last term on the right-hand side of \eqref{second-order error estimate13} can be estimated by 
\begin{equation}\label{second-order error estimate22}
\aligned
&M(\Delta_hT_3^{n+1/2},e_{\psi}^{n+1/2})_m\leq C\|e_{\psi}^{n+1/2}\|_m^2+C\|\phi\|^2_{W^{2,\infty}(J;W^{4,\infty}(\Omega))}\Delta t^4\\
&\ \ \ \ \ +C(\|\phi\|^2_{L^{\infty}(J;W^{8,\infty}(\Omega))}+\|\phi\|^2_{L^{\infty}(J;W^{6,\infty}(\Omega))})(h_x^4+h_y^4).
\endaligned
\end{equation} 
Combining \eqref{second-order error estimate11} with the above equations leads to
\begin{equation}\label{second-order error estimate23}
\aligned
&\frac{\|e_{\psi}^{n+1}\|_m^2-\|e_{\psi}^{n}\|_m^2}{2\Delta t}+\beta\|e_{\psi}^{n+1/2}\|_m^2+M\alpha\frac{\|e_{\phi}^{n+1}\|_m^2-\|e_{\phi}^{n}\|_m^2
}{2\Delta t}\\
&+M\frac{\|\nabla_h(\Delta_he_{\phi}^{n+1})\|^2-\|\nabla_h(\Delta_he_{\phi}^{n})\|^2
}{2\Delta t}\\
&+\frac{M}{2\Delta t}(\|\Delta_he_{\phi}^{n+1}-\Delta_he_{\phi}^{n}\|_m^2-\|\Delta_he_{\phi}^{n}-\Delta_he_{\phi}^{n-1}\|_m^2)\\
\leq &C(\|e_{\phi}^n\|_m^2+\|\nabla_h(\Delta_he_{\phi}^n)\|^2)+C(\|e_{\phi}^{n-1}\|_m^2+\|\nabla_h(\Delta_he_{\phi}^{n-1})\|^2)\\
&+C\|e_{\psi}^{n+1/2}\|_m^2+C(e_r^{n+1/2})^2+\frac{M}{\Delta t}(\|\Delta_he_{\phi}^{n+1}\|_m^2-\|\Delta_he_{\phi}^{n}\|_m^2)\\
&+C(\|\phi\|^2_{L^{\infty}(J;W^{8,\infty}(\Omega))}+\|\phi\|^2_{L^{\infty}(J;W^{6,\infty}(\Omega))})(h_x^4+h_y^4)\\
&+C(\|\phi\|^2_{W^{4,\infty}(J;L^{\infty}(\Omega))}+\|\phi\|^2_{W^{2,\infty}(J;W^{4,\infty}(\Omega))})\Delta t^4.
\endaligned
\end{equation}
Next we give the error estimate of auxiliary function $r$.
Multiplying \eqref{second-order error estimate9} by $e_{r}^{n+1}+e_{r}^{n}$ leads to 
\begin{equation}\label{second-order error estimate24}
\aligned
\frac{(e_r^{n+1})^2-(e_r^{n})^2}{\Delta t}=&
\frac{e_r^{n+1/2}}{\sqrt{E_1^h(\tilde{Z}^{n+1/2})}}(F^{\prime}(\tilde{Z}^{n+1/2}),
d_tZ^{n+1})_m\\
&-\frac{e_r^{n+1/2}}{\sqrt{E_1(\phi^{n+1/2})}}\int_{\Omega}F^{\prime}(\phi^{n+1/2})\phi^{n+1/2}_t d\textbf{x}\\
&+T_{4}^{n+1/2}\cdot (e_{r}^{n+1}+e_{r}^{n}).
\endaligned
\end{equation}
The first two terms on the right-hand side of \eqref{second-order error estimate24} can be transformed into:
\begin{equation}\label{second-order error estimate25}
\aligned
&\frac{e_r^{n+1/2}}{\sqrt{E_1^h(\tilde{Z}^{n+1/2})}}(F^{\prime}(\tilde{Z}^{n+1/2}),
d_tZ^{n+1})_m-\frac{e_r^{n+1/2}}{\sqrt{E_1(\phi^{n+1/2})}}\int_{\Omega}F^{\prime}(\phi^{n+1/2})\phi^{n+1/2}_t d\textbf{x}\\
=&\frac{e_{r}^{n+1/2}}{\sqrt{E_1(\phi^{n+1/2})}}\left((F^{\prime}(\phi^{n+1/2}),d_t\phi^{n+1})_m-\int_{\Omega}F^{\prime}(\phi^{n+1/2})\phi^{n+1/2}_t d\textbf{x}\right)\\
&+e_{r}^{n+1/2}(\frac{F^{\prime}(\tilde{Z}^{n+1/2})}{\sqrt{E_1^h(\tilde{Z}^{n+1/2})}}-
\frac{F^{\prime}(\phi^{n+1/2})}{\sqrt{E_1(\phi^{n+1/2})}},d_t\phi^{n+1})_m\\
&+\frac{e_{r}^{n+1/2}}{\sqrt{E_1^h(\tilde{Z}^{n+1/2})}}(F^{\prime}(\tilde{Z}^{n+1/2}),
d_te_{\phi}^{n+1})_m,
\endaligned
\end{equation}
which can be handled in a similar way as in \cite{li2019energy}. Thus we have
\begin{equation}\label{second-order error estimate26}
\aligned
&\frac{e_r^{n+1/2}}{\sqrt{E_1^h(\tilde{Z}^{n+1/2})}}(F^{\prime}(\tilde{Z}^{n+1/2}),
d_tZ^{n+1})_m-\frac{e_r^{n+1/2}}{\sqrt{E_1(\phi^{n+1/2})}}\int_{\Omega}F^{\prime}(\phi^{n+1/2})\phi^{n+1/2}_t d\textbf{x}\\
\leq &C(e_r^{n+1/2})^2+C\|\phi\|^2_{W^{1,\infty}(J;L^{\infty}(\Omega))}(\|e_{\phi}^{n}\|_m^2+\|e_{\phi}^{n-1}\|_m^2)\\
&+C\|e_{\psi}^{n+1/2}\|_m^2+
C\|\phi\|^2_{W^{1,\infty}(J;W^{2,\infty}(\Omega))}(h_x^4+h_y^4).
\endaligned
\end{equation}
Substituting \eqref{second-order error estimate26} into \eqref{second-order error estimate24} and applying Cauchy-Schwartz inequality, we can obtain
\begin{equation}\label{second-order error estimate27}
\aligned
\frac{(e_r^{n+1})^2-(e_r^{n})^2}{\Delta t}\leq&C(e_r^{n+1/2})^2+C\|\phi\|^2_{W^{1,\infty}(J;L^{\infty}(\Omega))}(\|e_{\phi}^{n}\|_m^2+\|e_{\phi}^{n-1}\|_m^2)\\
&+C\|e_{\psi}^{n+1/2}\|_m^2+
C\|\phi\|^2_{W^{1,\infty}(J;W^{2,\infty}(\Omega))}(h_x^4+h_y^4)\\
&+C\|r\|^2_{W^{3,\infty}(J)}\Delta t^4.
\endaligned
\end{equation}
Combining \eqref{second-order error estimate23} with \eqref{second-order error estimate27} and multiplying by $2\Delta t$, summing over $n,~n=0,1,\ldots,k$, we have
\begin{equation}\label{second-order error estimate28}
\aligned
&\|e_{\psi}^{k+1}\|_m^2+\beta\sum_{n=0}^{k}\Delta t\|e_{\psi}^{n+1/2}\|^2+M\alpha\|e_{\phi}^{k+1}\|^2+M\|\nabla_h(\Delta_he_{\phi}^{k+1})\|^2+2(e_r^{k+1})^2\\
\leq &2M\|\Delta_he_{\phi}^{k+1}\|_m^2+C\sum_{n=0}^{k}\Delta t\|e_{\phi}^n\|_m^2+C\sum_{n=0}^{k}\Delta t\|\nabla_h(\Delta_he_{\phi}^n)\|_m^2+C\sum_{n=0}^{k}\Delta t\|e_{\psi}^{n}\|^2\\
&+C\sum_{n=0}^{k}\Delta t(e_r^n)^2+C(\|\phi\|^2_{L^{\infty}(J;W^{8,\infty}(\Omega))}+\|\phi\|^2_{L^{\infty}(J;W^{6,\infty}(\Omega))})(h_x^4+h_y^4)\\
&+C(\|\phi\|^2_{W^{4,\infty}(J;L^{\infty}(\Omega))}+\|\phi\|^2_{W^{2,\infty}(J;W^{4,\infty}(\Omega))})\Delta t^4.
\endaligned
\end{equation}
To carry out further analysis, we should give the following inequality first.
Applying
$e_{\phi}^k=e_{\phi}^0+\sum\limits_{l=1}^{k}\Delta te_{\psi}^l$
and using Cauchy-Schwarz inequality, we obtain that
\begin{equation}\label{second-order error estimate29}
\aligned
\|e_{\phi}^k\|_m^2\leq &2\|e_{\phi}^0\|_m^2+2\|\sum\limits_{l=1}^{k}\Delta te_{\psi}^l \|_m^2\\
\leq &2\|e_{\phi}^0\|_m^2+2T\sum\limits_{l=1}^{k}\Delta t\|e_{\psi}^l\|_m^2.
\endaligned
\end{equation}
Applying Lemma \ref{le: control of backward diffusion term} and \eqref{second-order error estimate29}, the first term on the right-hand side of \eqref{second-order error estimate28} can be transformed into
\begin{equation}\label{second-order error estimate30}
\aligned
2M\|\Delta_he_{\phi}^{k+1}\|_m^2\leq &C\|e_{\phi}\|_m^2+\frac{M}{2}\|\nabla_h(\Delta_he_{\phi}^{k+1})\|^2\\
\leq&C\sum\limits_{l=1}^{k}\Delta t\|e_{\psi}^l\|_m^2+\frac{M}{2}\|\nabla_h(\Delta_he_{\phi}^{k+1})\|^2.
\endaligned
\end{equation} 
Then using the discrete Gronwall inequality and Lemma \ref{le: control of backward diffusion term},  \eqref{second-order error estimate28} can be estimated as follows:
\begin{equation}\label{second-order error estimate31}
\aligned
&\|e_{\psi}^{k+1}\|_m^2+\|e_{\phi}^{k+1}\|_m^2+\|\Delta_he_{\phi}^{k+1}\|_m^2+
\|\nabla_h(\Delta_he_{\phi}^{k+1})\|^2+(e_r^{k+1})^2\\
\leq &C(\|\phi\|^2_{W^{4,\infty}(J;L^{\infty}(\Omega))}+\|\phi\|^2_{W^{2,\infty}(J;W^{4,\infty}(\Omega))})\Delta t^4\\
&+C\|\phi\|^2_{L^{\infty}(J;W^{8,\infty}(\Omega))}(h_x^4+h_y^4),\quad\forall 0\le k\le N-1.
\endaligned
\end{equation}
It remains to verify the hypothesis \eqref{second-order error estimate17_hypotheses}. Actually this part of the proof follows a similar procedure as  in our previous works \cite{li2018block,li2019energy}. For the readers' convenience, we still provide a detail proof for $\|Z^n\|_{{L^{\infty}(\Omega)}} \leq C_1$ in the following two steps by using the mathematical induction.

\textit{\textbf{Step 1}} (Definition of $C_1$):  
Using the scheme \eqref{e_second-order1}-\eqref{e_second-order4} for $n=0$ and applying the inverse assumption, we can get the approximation $Z^1$ with the following property:
\begin{equation*}
\aligned
& \|Z^1\|_{{L^{\infty}(\Omega)}} \leq \|Z^1-\phi^1\|_{{L^{\infty}(\Omega)}} +\|\phi^1\|_{{L^{\infty}(\Omega)}}  \\
\leq&  \|Z^1-\Pi_h\phi^1\|_{{L^{\infty}(\Omega)}} +\|\Pi_h\phi^1-\phi^1\|_{{L^{\infty}(\Omega)}} +\|\phi^1\|_{{L^{\infty}(\Omega)}} \\
\leq&Ch^{-1}(\|Z^1-\phi^1 \|_m+\|\phi^1-\Pi_h\phi^1\|_m)+\|\Pi_h\phi^1-\phi^1\|_{{L^{\infty}(\Omega)}} +\|\phi^1\|_{{L^{\infty}(\Omega)}} \\
\leq&C(h+h^{-1}\Delta t^2)+\|\phi^1\|_{{L^{\infty}(\Omega)}} \leq C.
\endaligned
\end{equation*}
where $h=\max\{h_x,h_y\}$ and $\Pi_h$ is an bilinear interpolant operator with the following estimate:
\begin{equation}\label{e_boundedness_added2}
\aligned
\|\Pi_h\phi^1-\phi^1\|_{{L^{\infty}(\Omega)}} \leq Ch^2.
\endaligned
\end{equation}
 Thus we can choose the positive constant $C_1$ independent of $h$ and $\Delta t$ such that
\begin{align*}
C_1&\geq \max\{\|Z^{1}\|_{{L^{\infty}(\Omega)}} , 2\|\phi(t^n)\|_{{L^{\infty}(\Omega)}} \}.
\end{align*}

\textit{\textbf{Step 2}} (Induction): By the definition of $C_1$, it is trivial that hypothesis $\|Z^l\|_{{L^{\infty}(\Omega)}}\leq C_1$ holds true for $l=1$. Supposing that $\|Z^{l-1}\|_{{L^{\infty}(\Omega)}} \leq C_1$ holds true for an integer $l=1,\cdots,k+1$, with the aid of the estimate \eqref{second-order error estimate31}, we have that
$$\|Z^{l}-\phi^l\|_m\leq C(\Delta t^2+h^2).$$
Next we prove that $\|Z^{l}\|_{{L^{\infty}(\Omega)}} \leq C_1$ holds true.
Since
\begin{equation}\label{e_boundedness_added3}
\aligned
& \|Z^l\|_{{L^{\infty}(\Omega)}} \leq \|Z^l-\phi^l\|_{{L^{\infty}(\Omega)}} +\|\phi^l\|_{{L^{\infty}(\Omega)}} \\
\leq & \|Z^l-\Pi_h\phi^l\|_{{L^{\infty}(\Omega)}} +\|\Pi_h\phi^l-\phi^l\|_{{L^{\infty}(\Omega)}} +\|\phi^l\|_{{L^{\infty}(\Omega)}} \\
\leq&Ch^{-1}(\|Z^l-\phi^l \|_m+\|\phi^l-\Pi_h\phi^l\|_m)+\|\Pi_h\phi^l-\phi^l\|_{{L^{\infty}(\Omega)}} +\|\phi^l\|_{{L^{\infty}(\Omega)}} \\
\leq&C_4(h+h^{-1}\Delta t^2)+\|\phi^1\|_{{L^{\infty}(\Omega)}} .
\endaligned
\end{equation}
Let $\Delta t\leq C_5h$ and a positive constant $h_1$ be small enough to satisfy
$$C_4(1+C_5^2)h_1\leq\frac{C_1}{2}.$$
Then for $h\in (0,h_1],$ we derive from \eqref{e_boundedness_added3} that
\begin{equation*}
\aligned
\|Z^l\|_{{L^{\infty}(\Omega)}} 
\leq&C_4(h+h^{-1}\Delta t^2)+\|\phi^l\|_{{L^{\infty}(\Omega)}} \\
\leq &C_4(h_1+C_5^2h_1)+\frac{C_1}{2}
\leq C_1.
\endaligned
\end{equation*}
This indicates that $\|Z^n\|_{{L^{\infty}(\Omega)}} \leq C_1$ for all $n$.  The proof for the other two inequalities in \eqref{second-order error estimate17_hypotheses}
is essentially identical with the above procedure so we skip it for the sake of brevity.
\end{proof}

%==================================================================
\section{Numerical results and discussions}
In this section, we carry out some numerical experiments with the proposed scheme for the MPFC equation. We first verify the order of convergence. 
 Then we plot evolutions of the original energy as well as the pseudo energy to show that the pseudo energy is indeed dissipative while the original energy is not.

%==================================================================
\subsection{Accuracy tests}
We take $\Omega=(0,1)\times(0,1)$, $T=0.5$, $\epsilon=0.25$, $\beta=0.9$,  $M=0.001$ and the initial solution $\phi_0=\cos(2\pi x)\cos(2\pi y)$ with the homogenous Neumann boundary conditions. We use the second-order scheme \eqref{e_second-order1}-\eqref{e_second-order4} and measure the Cauchy error since we do not know the exact solution. Specifically, the error between two different grid spacings $h$ and $\frac{h}{2}$ is calculated by $\|e_{\zeta}\|=\|\zeta_h-\zeta_{h/2}\|$. We take the time step to be $\Delta t=\frac{T}{N}$ with $N=N_x=N_y$, and list the results in Table \ref{table2_example1}.
 For simplicity, we define $\|e_{f}\|_{\infty}=\max\limits_{0\leq l\leq k}\|e_{f}^{l}\|$. We observe a solid second order convergence rate,  which are consistent with the error estimates in Theorem \ref{thm: second-order error estimate}. 

\begin{table}[htbp]
\renewcommand{\arraystretch}{1.1}
\small
\centering
\caption{Errors and convergence rates for the scheme \eqref{e_second-order1}-\eqref{e_second-order4}.}\label{table2_example1}
\begin{tabular}{p{1.5cm}p{1.5cm}p{0.7cm}p{1.8cm}p{0.7cm}p{1.5cm}p{0.7cm}}\hline
$N_x\times N_y$    &$\|e_{\phi}\|_{\infty,m}$    &Rate &$\|\nabla_h(\Delta_he_{\phi})\|_{\infty}$   &Rate  
&$\|e_{r}\|_{\infty}$    &Rate   \\ \hline
$20\times 20$    &1.15E-1                & ---    &79.6E-0         &---  &2.15E-2         &---\\
$40\times 40$    &3.15E-2                &1.87    &22.0E-0         &1.85 &6.62E-3         &1.70\\
$80\times 80$    &8.02E-3                &1.97     &5.62E-0         &1.97 &1.28E-3         &2.38\\
$160\times 160$    &2.11E-3          &1.93    &1.48E-0         &1.93   &2.32E-4         &2.46\\
\hline
\end{tabular}
\end{table}
%==================================================================
\subsection{Energy stability test}
In this example, we set  $\Omega=(0,128)\times(0,128)$,
$M=1$, $\epsilon=0.025$, $\beta=0.1$, and consider the MPFC model with the periodic boundary conditions. The initial condition is taken as follows \cite{baskaran2013energy,guo2018high}:
\begin{equation}\label{numerical_initial}
\aligned
\phi_0(x,y)=&0.07-0.02\cos(\frac{2\pi(x-12)}{32})\sin(\frac{2\pi(y-1)}{32})+0.02\cos^2(\frac{\pi(x+10)}{32})\\
&\cos^2(\frac{\pi(y+3)}{32})-0.01\sin^2(\frac{4\pi x}{32})\sin^2(\frac{4\pi(y-6)}{32}).
\endaligned
\end{equation}
We take $\Delta t=0.05$ and evolve the system to the final time $T=10$. The evolutions of discrete original energy and pseudo energy using the second-order scheme are plotted in Figure \ref{fig: pseudo energy}. We  observe that the discrete original energy may increase on some time intervals, while the pseudo energy are non-increasing at all times, which is consistent with our analysis.

\begin{figure}[!htp]
\centering
\includegraphics[scale=0.50]{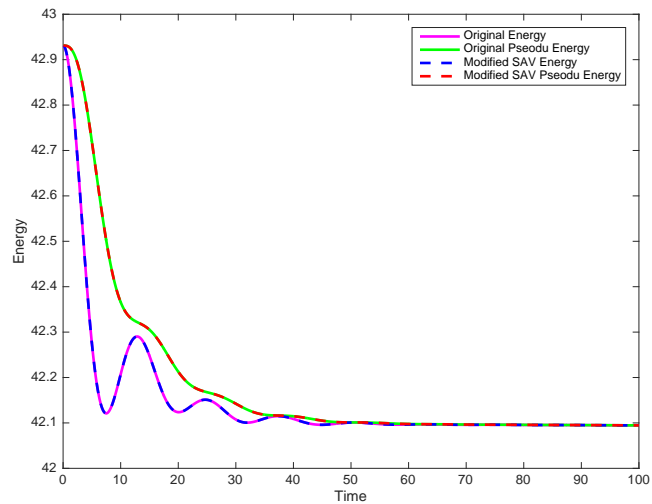}
\caption{The discrete original energy and pseudo energy plotted as functions of time}  \label{fig: pseudo energy}
\end{figure}
%==================================================================
\begin{comment}
\subsection{Long time simulation}
As a final example, we show a long time simulation of the MPFC model with the homogenous Neumann boundary conditions.We set $\Omega=(0,128)\times(0,128)$ with a random initial data $\phi_{i,j}=\phi_0+\eta_{i,j}$, where $\phi_0=0.1$ and $\eta_{i,j}$ is a uniformly distributed random number satisfying $|\eta_{i,j}|\leq 0.1$. The other parameters are $M=1$, $\epsilon=0.025$, $\beta=0.5$. We present the evolution of the density field $\phi$  using the second-order scheme with $\Delta t=1$ and $h=1$ in Figure \ref{fig: long time simulation}. We observe that the random initial profile evolves into two regions separated by an interfacial region. 

\begin{figure}[!htp]
\centering
\includegraphics[scale=0.30]{long_t500.eps}
\includegraphics[scale=0.30]{long_t900.eps}
\includegraphics[scale=0.30]{long_t1600.eps}
\includegraphics[scale=0.30]{long_t2400.eps}
\caption{The evolution of the density field $\phi$ calculated using the second-order scheme at $t=500$, $900$, $1600$, $2400$, respectively.}  \label{fig: long time simulation}
\end{figure}
\end{comment}
%==================================================================
\subsection{Summary}
We constructed in this paper two efficient schemes 
for the MPFC model based on the SAV approach and block finite-difference method. Since the original energy of the MPFC equation may increase in time on some time intervals, we introduced a pseudo energy that is dissipative for all times. It is shown that our schemes conserve mass and are unconditionally energy stable with respect to the pseudo energy. We also established rigorously 
 second-order error estimates in both time and space for our  second-order  SAV block-centered finite difference method. Finally some numerical experiments are presented to validate our theoretical results.

%==================================================================
%\paragraph{ Acknowledgements}
%    The authors express their thanks to the two referees for their helpful suggestions, which lead to improvements
%   of the presentation.
%==================================================================

% \section*{References}
\bibliographystyle{siamplain}
\bibliography{SAV_MPFC}

\end{document}